\documentclass[12pt]{amsart}

\usepackage{amsmath}
\usepackage{amscd}
\usepackage{amssymb}
\usepackage{amsfonts}
\usepackage{amsthm}
\usepackage{enumerate}
\usepackage{hyperref}
\usepackage{color}
\usepackage{graphicx}

\theoremstyle{plain}
\newtheorem{thm}{Theorem}[section]

\newtheorem{lem}[thm]{Lemma}

\theoremstyle{definition}

\newtheorem{dfns-rems}[thm]{Definitions and Remarks}
\newtheorem{notas-rems}[thm]{Notations and Remarks}
\newtheorem{exmps-rems}[thm]{Examples and Remarks}

\DeclareMathOperator{\ind-match}{ind-match}

\begin{document}


\title[Regularity of symbolic powers of Cameron-Walker graphs]{Regularity of symbolic powers of edge ideals of Cameron-Walker graphs}


\author[S. A. Seyed Fakhari]{S. A. Seyed Fakhari}

\address{S. A. Seyed Fakhari, School of Mathematics, Statistics and Computer Science,
College of Science, University of Tehran, Tehran, Iran, and Institute of Mathematics, Vietnam Academy of Science and Technology, 18 Hoang Quoc Viet, Hanoi, Vietnam.}

\email{aminfakhari@ut.ac.ir}

\begin{abstract}
A Cameron-Walker graph is a graph for which the matching number and the induced matching number are the same. Assume that $G$ is a Cameron-Walker graph with edge ideal $I(G)$, and let $\ind-match(G)$ be the induced matching number of $G$. It is shown that for every integer $s\geq 1$, we have the equality ${\rm reg}(I(G)^{(s)})=2s+\ind-match(G)-1$, where $I(G)^{(s)}$ denotes the $s$-th symbolic power of $I(G)$.
\end{abstract}






\thanks{This research is partially funded by the Simons Foundation Grant Targeted for Institute of Mathematics, Vietnam Academy of Science and Technology.}


\maketitle


\section{Introduction} \label{sec1}

Let $\mathbb{K}$ be a field and $S = \mathbb{K}[x_1,\ldots,x_n]$  be the
polynomial ring in $n$ variables over $\mathbb{K}$. Suppose that $M$ is a nonzero finitely generated graded $S$-module with minimal free resolution
$$0  \longrightarrow \cdots \longrightarrow  \bigoplus_{j}S(-j)^{\beta_{1,j}(M)} \longrightarrow \bigoplus_{j}S(-j)^{\beta_{0,j}(M)}   \longrightarrow  M \longrightarrow 0.$$
The Castelnuovo-Mumford regularity (or simply, regularity) of $M$, denoted by ${\rm reg}(M)$, is defined as
$${\rm reg}(M)={\rm max}\{j-i\mid \beta _{i,j}(M)\neq 0\}.$$

Let $G$ be a graph with vertex set $V(G)=\big\{x_1, \ldots, x_n\big\}$ and edge set $E(G)$. The {\it edge ideal} of $G$, denoted by $I(G)$, is the monomial ideal of $S$ which is generated by quadratic squarefree monomials corresponding to edges of $G$, i.e.,$$I(G)=\big(x_ix_j: x_ix_j\in E(G)\big).$$Computing and finding bounds for the regularity of edge ideals and their powers have been studied by a number of researchers (see for example \cite{ab},  \cite{abs}, \cite{b}, \cite{bbh1}, \cite{bbh}, \cite{bht}, \cite{dhs}, \cite{ha}, \cite{jns}, \cite{k}, \cite{khm}, \cite{msy}, \cite{sy} and \cite{wo}).

The study of the regularity of symbolic powers of edge ideals has been started by Gu, H${\rm \grave{a}}$, O'Rourke and Skelton \cite{ghos}. They proved that for every graph $G$ with induced matching number $\ind-match(G)$, and for every integer $s\geq 1$, we have $${\rm reg}(I(G)^{(s)})\geq 2s+\ind-match(G)-1,$$where $I(G)^{(s)}$ denotes the $s$-th symbolic power of $I(G)$, \cite[Theorem 4.6]{ghos}. We mention that the above inequality in the special case of $s=1$ was proved by Katzman \cite{k}. Also, the above inequality is true if one replaces the symbolic power with ordinary power and it is proved by Beyarslan, H${\rm \grave{a}}$ and Trung, \cite[Theorem 4.5]{bht}.

N. C. Minh Conjectured that for every graph $G$ and for any integer $s\geq 1$, the regularity of the $s$-th ordinary and symbolic powers of $I(G)$ are equal, i.e., $${\rm reg}(I(G)^{(s)})={\rm reg}(I(G)^s)$$ (see \cite{ghos}). We know from \cite[Theorem 5.9]{svv} that for any bipartite graph $G$, the equality $I(G)^{(s)}=I(G)^s$ holds for any $s\geq 1$. In particular, Minh's conjecture is trivially true for bipartite graphs. Gu, H${\rm \grave{a}}$, O'Rourke and Skelton \cite{ghos} verified Minh's conjecture for any cycle graphs. In \cite{s8}, we proved this conjecture for every unicyclic graph. Jayanthan and Kumar \cite{jk} showed that Minh's conjecture is true for some classes of graphs which are obtained by the clique sum of odd cycles and bipartite graphs.

In \cite{s11}, we computed the regularity of symbolic powers of edge ideals of chordal graphs. More precisely, it is proven in \cite[Theorem 3.3]{s11} that for every chordal graph $G$ and every integer $s\geq 1$, the equality$${\rm reg}(I(G)^{(s)})=2s+\ind-match(G)-1$$holds.

The goal of this paper is to compute the regularity of symbolic powers of edge ideals of the so-called {\it Cameron-Walker} graphs. We recall that a graph is said to be a Cameron-Walker graph if it has the same matching number and induced matching number. The reason for this naming is that the structure of theses graphs has been determined by Cameron and Walker \cite{cw}. Indeed, it is clear that a graph is Cameron-Walker if and only if all its connected components are Cameron-Walker. By \cite[Theorem 1]{cw} (see also \cite[Remark 0.1]{hhko}), a connected graph $G$ is a Cameron-Walker graph if and only if

$\bullet$ it is a star graph, or

$\bullet$ it is a star triangle, or

$\bullet$ it consists of a connected bipartite graph $H$ with vertex partition $V(H)=X\cup Y$ with the property that  there is at least one pendant edge attached  to each  vertex of $X$ and there may be some pendant triangles attached to each vertex of $Y$.

Banerjee, Beyarslan and H${\rm \grave{a}}$ \cite[Corollary 3.5]{bbh} prove that for every Cameron-Walker $G$ and any integer $s\geq1$,$${\rm reg}(I(G)^s)=2s+\ind-match(G)-1.$$As the main result of this paper, it is shown in Theorem \ref{main} that the above equality is true if one replaces the ordinary power with symbolic power. This, in particular, implies that Minh's conjecture is true for any Cameron-Waler graph.


\section{Preliminaries} \label{sec2}

In this section, we provide the definitions and basic facts which will be used in the next section.

Let $G$ be a simple graph with vertex set $V(G)=\big\{x_1, \ldots,
x_n\big\}$ and edge set $E(G)$. For a vertex $x_i$, its {\it degree}, denoted by ${\rm deg}_G(x_i)$, is the number of edges of $x_i$ which are incident to $x_i$. A vertex of degree one is a {\it leaf} and the unique edge incident to a leaf is called a {\it pendant edge}. A {\it pendant triangle} of $G$ is a triangle $T$ of $G$, with the property that exactly two vertices of $T$ have degree two in $G$. For every subset $U\subset V(G)$, the graph $G-U$ has vertex set $V(G-U)=V(G)\setminus U$ and edge set $E(G-U)=\{e\in E(G)\mid e\cap U=\emptyset\}$. When $U=\{x\}$ is a singleton, we write $G\setminus x$ instead of $G\setminus \{x\}$. A subgraph $H$ of $G$ is called {\it induced} provided that two vertices of $H$ are adjacent if and only if they are adjacent in $G$. A graph $G$ is called {\it chordal} if it has no induced cycle of length at least four. A subset $C$ of $V(G)$ is a {\it vertex cover} of $G$ if every edge of $G$ is incident to at least one vertex of $C$. A vertex cover $C$ is a {\it minimal vertex cover} if no proper subset of $C$ is a vertex cover of $G$. The set of minimal vertex covers of $G$ will be denoted by $\mathcal{C}(G)$.

For every subset $C$ of $\big\{x_1, \ldots, x_n\big\}$, we denote by $\mathfrak{p}_C$, the monomial prime ideal which is generated by the variables belong to $C$. It is well-known that for every graph $G$,$$I(G)=\bigcap_{C\in \mathcal{C}(G)}\mathfrak{p}_C.$$

Let $G$ be a graph. A subset $M\subseteq E(G)$ is a {\it matching} if $e\cap e'=\emptyset$, for every pair of edges $e, e'\in M$. The cardinality of the largest matching of $G$ is called the {\it matching number} of $G$ and is denoted by ${\rm match}(G)$. A matching $M$ of $G$ is an {\it induced matching} of $G$ if for every pair of edges $e, e'\in M$, there is no edge $f\in E(G)\setminus M$ with $f\subset e\cup e'$. The cardinality of the largest induced matching of $G$ is the {\it induced  matching number} of $G$ and is denoted by $\ind-match(G)$. As we mentioned in Section \ref{sec1}, a graph $G$ is a Cameron-Walker graph if ${\rm match}(G)=\ind-match(G)$.

Let $I$ be an ideal of $S$ and let ${\rm Min}(I)$ denote the set of minimal primes of $I$. For every integer $s\geq 1$, the $s$-th {\it symbolic power} of $I$,
denoted by $I^{(s)}$, is defined to be$$I^{(s)}=\bigcap_{\frak{p}\in {\rm Min}(I)} {\rm Ker}(S\rightarrow (S/I^s)_{\frak{p}}).$$We set $I^{(s)}=S$, for any integer $s\leq 0$.

Let $I$ be a squarefree monomial ideal with the irredundant
primary decomposition $$I=\frak{p}_1\cap\ldots\cap\frak{p}_r.$$It follows from \cite[Proposition 1.4.4]{hh} that for every integer $s\geq 1$, $$I^{(s)}=\frak{p}_1^s\cap\ldots\cap\frak{p}_r^s.$$In particular, for every graph $G$ and any integer $s\geq 1$, we have$$I(G)^{(s)}=\bigcap_{C\in \mathcal{C}(G)}\mathfrak{p}_C^s.$$


\section{Main results} \label{sec3}

In this section, we prove our main result, Theorem \ref{main}. The proof is based on an inductive argument and in order to use induction on power, we need the following lemma.

\begin{lem} \label{colon}
Let $G$ be a graph and assume that $T$ is a triangle of $G$, with vertex set $V(T)=\{x_1, x_2, x_3\}$. Suppose that ${\rm deg}_G(x_3)=1$. Then for every integer $s\geq 1$,$$\big(I(G)^{(s)}: x_1x_2x_3)=I(G)^{(s-2)}.$$
\end{lem}

\begin{proof}
For any minimal vertex cover $C$ of $G$, we have $V(T) \nsubseteq C$, because otherwise $C\setminus \{x_3\}$ would be vertex cover of $G$ which is properly contained in $C$. On the other hand, It is obvious that for every vertex cover $C$ of $G$, we have $|C\cap V(T)| \geq 2$. Therefore, $|C\cap V(T)|=2$, for every minimal vertex cover $C\in \mathcal{C}(G)$. Hence,$$\big(\mathfrak{p}_C^s: x_1x_2x_3\big)=\mathfrak{p}_C^{s-2},$$for every $C\in \mathcal{C}(G)$.  Since$$I(G)^{(s)}=\bigcap_{C\in \mathcal{C}(G)}\mathfrak{p}_C^s,$$we conclude that$$\big(I(G)^{(s)}: x_1x_2x_3\big)=\bigcap_{C\in \mathcal{C}(G)}\big(\mathfrak{p}_C^s: x_1x_2x_3\big)=\bigcap_{C\in \mathcal{C}(G)}\mathfrak{p}_C^{s-2}=I(G)^{(s-2)}.$$
\end{proof}

We are now ready to prove the main result of this paper.

\begin{thm} \label{main}
Let $G$ be a Cameron-Walker graph. Then for every integer $s\geq 1$, we have$${\rm reg}(I(G)^{(s)})=2s+\ind-match(G)-1.$$
\end{thm}

\begin{proof}
By \cite[Theorem 4.6]{ghos}, it is enough to show that$${\rm reg}(I(G)^{(s)})\leq 2s+\ind-match(G)-1.$$Without lose of generality, assume that $G$ has no isolated vertex and suppose $V(G)=\{x_1, \ldots, x_n\}$. We use induction on $|E(G)|+s$. The assertion is well-known for for $s=1$. Thus, assume that $s\geq 2$. If $|E(G)|=1$, then $I(G)=(x_1x_2)$. Consequently, $I(G)^{(s)}=(x_1^sx_2^s)$ and$${\rm reg}(I(G)^{(s)})=2s=2s+\ind-match(G)-1.$$Hence, suppose $|E(G)|\geq 2$.

First assume that $G$ is a disconnected graph and suppose $G_1, \ldots, G_p$ ($p\geq 2$) are the connected components of $G$. Let $H$ denote the disjoint union of $G_1, \ldots, G_{p-1}$. Clearly,$$\ind-match(H)+\ind-match(G_p)=\ind-match(G).$$Since $H$ and $G_p$ are Cameron-Walker graphs, using the induction hypothesis, for every integer $k\leq s$ we have$${\rm reg}(I(H)^{(k)})\leq 2k+\ind-match(H)-1,$$and$${\rm reg}(I(G_p)^{(k)})\leq 2k+\ind-match(G_p)-1.$$We conclude from \cite[Theorem 5.11]{hntt} that$${\rm reg}(I(G)^{(s)})\leq 2s+\ind-match(G)-1.$$

We now assume that $G$ is a connected graph. The desired equality follows from \cite[Theorem 3.3]{s11} if $G$ is a chordal graph. Thus, assume that $G$ is not  chordal. In particular, $G$ is not a star or a star triangle graph. Hence, it consists of a connected bipartite graph $H$ with vertex partition $V(H)=X\cup Y$  such that there is at least one pendant edge attached to each vertex of $X$ and that there may be some pendant triangles attached to each vertex of $Y$. If $G$ has no triangle, then it is a bipartite graph and by \cite[Theorem 5.9]{svv}, we have $I(G)^{(s)}=I(G)^s$. Therefore, in this case, the assertion follows from \cite[Corollary 3.5]{bbh}. Hence, assume that $G$ has at least one triangle, say $T$. Suppose without loss of generality that $V(T)=\{x_1, x_2, x_3\}$ and that ${\rm deg}_G(x_2)={\rm deg}_G(x_3)=2$. Consider the following short exact sequence.
\begin{align*}
0 \longrightarrow \frac{S}{(I(G)^{(s)}:x_1)}(-1)\longrightarrow \frac{S}{I(G)^{(s)}}\longrightarrow \frac{S}{I(G)^{(s)}+(x_1)}\longrightarrow 0
\end{align*}
It follows that
\[
\begin{array}{rl}
{\rm reg}(I(G)^{(s)})\leq \max\big\{{\rm reg}(I(G)^{(s)}:x_1)+1, {\rm reg}(I(G)^{(s)},x_1)\big\}.
\end{array} \tag{1} \label{1}
\]
As $G\setminus x_1$ is a (disconnected) Cameron-Walker graph with$$\ind-match(G\setminus x_1)=\ind-match(G),$$we conclude from the induction hypothesis that
\[
\begin{array}{rl}
&{\rm reg}(I(G)^{(s)},x_1)={\rm reg}(I(G\setminus x_1)^{(s)},x_1)={\rm reg}(I(G\setminus x_1)^{(s)})\\ &\leq 2s+\ind-match(G\setminus x_1)-1\\ &=2s+\ind-match(G)-1.
\end{array} \tag{2} \label{9}
\]

Therefore, using inequalities (\ref{1}) and (\ref{9}), we only need to show that$${\rm reg}(I(G)^{(s)}:x_1)\leq 2s+\ind-match(G)-2.$$Consider the following short exact sequence.
\begin{align*}
0 \longrightarrow \frac{S}{(I(G)^{(s)}:x_1x_2)}(-1)\longrightarrow \frac{S}{(I(G)^{(s)}:x_1)}\longrightarrow \frac{S}{(I(G)^{(s)}:x_1)+(x_2)}\longrightarrow 0
\end{align*}
It follows that
\[
\begin{array}{rl}
{\rm reg}(I(G)^{(s)}:x_1)\leq \max\big\{{\rm reg}(I(G)^{(s)}:x_1x_2)+1, {\rm reg}\big((I(G)^{(s)}:x_1), x_2\big)\big\}.
\end{array} \tag{3} \label{2}
\]

\vspace{0.3cm}
{\bf Claim 1.} ${\rm reg}(I(G)^{(s)}:x_1x_2)\leq 2s+\ind-match(G)-3$.

\vspace{0.3cm}
{\it Proof of Claim 1.} Consider the following short exact sequence.
\begin{align*}
0 \longrightarrow \frac{S}{(I(G)^{(s)}:x_1x_2x_3)}(-1)\longrightarrow \frac{S}{(I(G)^{(s)}:x_1x_2)}\longrightarrow \frac{S}{(I(G)^{(s)}:x_1x_2)+(x_3)}\longrightarrow 0
\end{align*}
It follows that
\[
\begin{array}{rl}
{\rm reg}(I(G)^{(s)}:x_1x_2)\leq \max\big\{{\rm reg}(I(G)^{(s)}:x_1x_2x_3)+1, {\rm reg}\big((I(G)^{(s)}:x_1x_2), x_3\big)\big\}.
\end{array} \tag{4} \label{3}
\]
We conclude from Lemma \ref{colon} and the induction hypothesis that
\[
\begin{array}{rl}
&{\rm reg}(I(G)^{(s)}:x_1x_2x_3)={\rm reg}(I(G)^{(s-2)})\leq 2(s-2)+\ind-match(G)-1\\ &=2s+\ind-match(G)-5.
\end{array} \tag{5} \label{4}
\]
On the other hand,$${\rm reg}\big((I(G)^{(s)}:x_1x_2), x_3\big)={\rm reg}\big((I(G)^{(s)}, x_3):x_1x_2\big)={\rm reg}\big(I(G\setminus x_3)^{(s)}:x_1x_2\big).$$
Since $x_1x_2$ is a pendant edge of $G\setminus x_3$, it follows from \cite[Lemma 3.3]{s8} that$$\big(I(G\setminus x_3)^{(s)}:x_1x_2\big)=I(G\setminus x_3)^{(s-1)}.$$Therefore,$${\rm reg}\big((I(G)^{(s)}:x_1x_2), x_3\big)={\rm reg}\big(I(G\setminus x_3)^{(s-1)}\big).$$As $G\setminus x_3$ is an induced subgraph of $G$, using \cite[Corollary 4.5]{ghos}, we have$${\rm reg}\big(I(G\setminus x_3)^{(s-1)}\big)\leq {\rm reg}\big(I(G)^{(s-1)}\big),$$and it follows from the induction hypothesis that
\[
\begin{array}{rl}
&{\rm reg}\big((I(G)^{(s)}:x_1x_2), x_3\big)\leq {\rm reg}\big(I(G)^{(s-1)}\big)\leq 2(s-1)+\ind-match(G)-1\\ &=2s+\ind-match(G)-3.$$
\end{array} \tag{6} \label{5}
\]
Finally, the assertion of Claim 1 follows from inequalities (\ref{3}), (\ref{4}) and (\ref{5}).

\vspace{0.3cm}
{\bf Claim 2.} ${\rm reg}\big((I(G)^{(s)}:x_1), x_2\big)\leq 2s+\ind-match(G)-2$.

\vspace{0.3cm}
{\it Proof of Claim 2.} Consider the following short exact sequence.
\begin{align*}
0 \longrightarrow &\frac{S}{\big(\big((I(G)^{(s)}:x_1),x_2\big): x_3\big)}(-1)\longrightarrow \frac{S}{\big((I(G)^{(s)}:x_1),x_2\big)}\longrightarrow\\ &\frac{S}{\big((I(G)^{(s)}:x_1),x_2, x_3\big)}\longrightarrow 0
\end{align*}
It follows that
\[
\begin{array}{rl}
&{\rm reg}\big((I(G)^{(s)}:x_1), x_2\big)\leq\\ &\max\big\{{\rm reg}\big(\big((I(G)^{(s)}:x_1),x_2\big): x_3\big)+1, {\rm reg}\big((I(G)^{(s)}:x_1),x_2, x_3\big)\big\}.
\end{array} \tag{7} \label{6}
\]
Note that
\begin{align*}
& {\rm reg}\big(\big((I(G)^{(s)}:x_1),x_2\big): x_3\big)={\rm reg}\big((I(G)^{(s)},x_2): x_1x_3\big)={\rm reg}\big((I(G\setminus x_2)^{(s)},x_2): x_1x_3\big)\\ &={\rm reg}\big(\big(I(G\setminus x_2)^{(s)}: x_1x_3\big),x_2\big)={\rm reg}\big(I(G\setminus x_2)^{(s)}: x_1x_3\big).
\end{align*}
As $x_1x_3$ is a pendant edge of $G\setminus x_2$, it follows from \cite[Lemma 3.3]{s8} that$$\big(I(G\setminus x_2)^{(s)}:x_1x_3\big)=I(G\setminus x_2)^{(s-1)}.$$Therefore,$${\rm reg}\big(\big((I(G)^{(s)}:x_1),x_2\big): x_3\big)={\rm reg}\big(I(G\setminus x_2)^{(s-1)}\big).$$Since $G\setminus x_2$ is an induced subgraph of $G$, using \cite[Corollary 4.5]{ghos}, we have$${\rm reg}\big(I(G\setminus x_2)^{(s-1)}\big)\leq {\rm reg}\big(I(G)^{(s-1)}\big),$$and it follows from the induction hypothesis that
\[
\begin{array}{rl}
&{\rm reg}\big(\big((I(G)^{(s)}:x_1),x_2\big): x_3\big)\leq {\rm reg}\big(I(G)^{(s-1)}\big)\leq 2(s-1)+\ind-match(G)-1\\ &=2s+\ind-match(G)-3.
\end{array} \tag{8} \label{7}
\]

On the other hand,
\begin{align*}
& {\rm reg}\big((I(G)^{(s)}:x_1),x_2, x_3\big)={\rm reg}\big((I(G)^{(s)},x_2,x_3): x_1\big)\\ &={\rm reg}\big((I(G\setminus \{x_2,x_3\})^{(s)},x_2, x_3): x_1\big)\\ &={\rm reg}\big(\big(I(G\setminus \{x_2,x_3\})^{(s)}: x_1\big),x_2, x_3\big)\\ &={\rm reg}\big(I(G\setminus \{x_2,x_3\})^{(s)}: x_1\big).
\end{align*}

We know from \cite[Lemma 4.2]{s3} that$${\rm reg}\big(I(G\setminus \{x_2,x_3\})^{(s)}: x_1\big)\leq {\rm reg}\big(I(G\setminus \{x_2,x_3\})^{(s)}\big).$$Consequently,$${\rm reg}\big((I(G)^{(s)}:x_1),x_2, x_3\big)\leq {\rm reg}\big(I(G\setminus \{x_2,x_3\})^{(s)}\big).$$As $G\setminus \{x_2,x_3\}$ is a Cameron-Walker graph with$$\ind-match(G\setminus \{x_2,x_3\})=\ind-match(G)-1,$$ we conclude from the induction hypothesis that
\begin{align*}
& {\rm reg}\big(I(G\setminus \{x_2,x_3\})^{(s)}\big)\leq 2s+\ind-match(G\setminus \{x_2,x_3\})-1\\ &=2s+\ind-match(G)-2.
\end{align*}
Therefore,
\[
\begin{array}{rl}
{\rm reg}\big((I(G)^{(s)}:x_1),x_2, x_3\big)\leq 2s+\ind-match(G)-2.
\end{array} \tag{9} \label{8}
\]

It now follows from inequalities (\ref{6}), (\ref{7}) and (\ref{8}) that$${\rm reg}\big((I(G)^{(s)}:x_1), x_2\big)\leq 2s+\ind-match(G)-2,$$and this proves claim 2.

\vspace{0.3cm}
We deduce from Claims 1, 2, and inequality (\ref{2}) that$${\rm reg}(I(G)^{(s)}:x_1)\leq 2s+\ind-match(G)-2.$$Hence, using inequalities (\ref{1}) and $\ref{9}$, we have$${\rm reg}(I(G)^{(s)})\leq 2s+\ind-match(G)-1.$$The reverse inequality follows from \cite[Theorem 4.6]{ghos}, and this completes the proof.
\end{proof}





\end{document}